\documentclass[10pt]{amsart}
\usepackage {amssymb,latexsym}
\usepackage [all]{xy}
\newtheorem{theorem}{Theorem}[section]
\newtheorem{lemma}[theorem]{Lemma}
\newtheorem {proposition}[theorem]{Proposition}
\newtheorem {corollary} [theorem] {Corollary}

\theoremstyle{definition}
\newtheorem{definition}[theorem]{Definition}
\newtheorem{example}[theorem]{Example}

\newtheorem{remark}[theorem]{Remark}
\numberwithin {equation}{section}

\def\cs{{$C^{\ast}$}}
\def\ra{{\rightarrow}}

\def\q{{\mathbb{Q}}}
\def\z{{\mathbb{Z}}}
\def\n{{\mathbb{N}}}
\def\c{{\mathbb{C}}}
\def\r{{\mathbb{R}}}

\def\H{{\mathcal{H}}}
\def\gh{{(G,H)}}
\def\gk{{(G,K)}}
\def\f{{\varphi}}
\def\si{{\sigma}}
\def\gaa{{\Gamma}}
\def\ga{{\gamma}}

\def\wrt{{with respect to }}

\def\ba{{\backslash}}
\def\inv{{^{-1}}}

\begin{document}
\title{Various commensurability relations in Hecke pairs and property (RD)}

\author{Vahid Shirbisheh}
%\address{}
\email{shirbisheh@gmail.com}


%\date{\today}

\begin{abstract}In this paper we continue our study of property (RD) for Hecke pairs initiated in \cite{s1}. We study the permanence of property (RD) under different commensurability relations of subgroups in Hecke pairs.  As an application, we prove that if $H$ is a normal subgroup of a group $G$ and $K$ is a subgroup of $G$ commensurable to $H$, then the Hecke pair $(G,K)$ has (RD) if and only if the quotient group $G/H$ has (RD). This is used to investigate an infinite number of non-elementary examples of Hecke pairs with property (RD). In particular, we introduce a class of groups all whose subgroups are almost normal and all such Hecke pairs have property (RD). We also discuss property (RD) of certain Hecke pairs arising from group extensions.
\end{abstract}
\maketitle

%\tableofcontents
%%%%%%%%%%%%%%%%%%%%%%%%%%%%%%%%%%      Section        %%%%%%%%%%%%%%%%%%%%%%%%%%%%%%%%%%%%%%%%%%%%%%%%%%%%%%%%%
\section {Introduction}
\label{sec:intro}
In this paper, all groups are discrete. A subgroup $H$ of a group $G$ is called an almost normal subgroup of $G$ if every double coset of $H$ is a union of finitely many left cosets of $H$. In this case the pair $\gh$ is called a Hecke pair. One should not confuse our definition of almost normal subgroups, which is popular in the context of Hecke \cs-algebras, with the other definition which is popular in the literature of group theory. By the second definition, $H$ is an almost normal subgroup of $G$ if its normalizer is a subgroup in $G$ of finite index. In group theory, an almost normal subgroup of $G$ (according to our definition) is called a conjugate-commensurable subgroup of $G$. This is because every conjugate of $H$, like $gHg\inv$, is commensurable with $H$. The Hecke algebra $\H\gh$ associated to the Hecke pair $\gh$ is the vector space of all finite support complex functions on the set of double cosets of $H$ in $G$ equipped with a convolution like product similar to (\ref{f:regrep}). The (reduced) Hecke \cs-algebra associated to a Hecke pair $\gh$ is the completion of the image of the Hecke algebra $\H\gh$ under the regular representation $\lambda: \H(G,H) \ra \mathcal{B}(\ell^2(H\ba G))$ defined by
\begin{equation}
\label{f:regrep} \lambda(f)(\xi)(g):=(f*\xi)(g):=\sum_{h\in <H\ba G>} f(gh\inv) \xi(h),
\end{equation}
for all $f \in \H(G,H)$ and $\xi \in \ell^2(H\ba G)$. In noncommutative geometry, Hecke \cs-algebras first appeared in \cite{bc} to construct a \cs-dynamical system reviling the class field theory of the field $\q$ of rational numbers. They are also considered as a generalization of reduced group \cs-algebras. This is the motivation of our program started in \cite{s1} to extend tools and methods of noncommutative geometry, developed originally for reduced group \cs-algebras, to the more general setting of Hecke \cs-algebras. In \cite{s1}, we defined property (RD) for Hecke pairs, (see Definition \ref{def:rd} below), and studied a class of examples of Hecke pairs with this property. We also showed that when a Hecke pair possesses property (RD), the algebra of rapidly decreasing functions is a smooth subalgebra of the associated Hecke \cs-algebra, in other words, it is dense and stable under holomorphic functional calculus of the Hecke \cs-algebra. In this paper we continue our study of property (RD) for Hecke pairs. We refer the reader to \cite{s1, s3} for basic definitions, notations and results, where we also explained why property (RD) is important for developing noncommutative geometry over Hecke \cs-algebras. In order to give a big picture of our work in the present paper, we need to differ between different types of commensurability between groups and subgroups.
\begin{definition}
\label{def:comm}
\begin{itemize}
\item [(i)] Two subgroups $H$ and $K$ of a group $G$ are called commensurable if  there exists some $g\in G$ such that $H\cap gKg^{-1}$ is a finite index subgroup of both $H$ and $K$. The subgroups $H$ and $K$ are called strongly commensurable if $H\cap K$ has finite index in both $H$ and $K$.
\item [(ii)] Two groups $G_1$ and $G_2$ are called weakly commensurable if they have subgroups $H_1\leq G_1$ and $H_1\leq G_1$ of finite index such that $H_1$ and $H_2$ are isomorphic.
\end{itemize}
\end{definition}

We note that all the above three different definitions of commensurability relations have appeared in literatures under the same name, see for example \cite{harpe, krieg, luck, shimura}, and the terms ``weakly'' and ``strongly'' are temporary terminologies used in this paper. For another definition of commensurability see Definition 11.42 of \cite{meier}, which is equivalent to weak commensurability. The notion of strong commensurability and almost normal subgroups are closely related. Let $H$ be a subgroup of a group $G$. The commensurator of $H$ in $G$ is the set
\[ \tilde{H}:= \{ g\in G; H \,\text{and}\, gHg^{-1} \,\text{are strongly commensurable} \},
\]
which happens to be a subgroup of $G$, \cite{shimura}. It follows immediately from the definition that  $H$ is almost normal in $G$ if and only if $\tilde{H}=G$.

It is easily seen that all the above commensurability relations are equivalence relations and it is desirable to investigate the behavior of different properties of groups or pairs of groups and subgroups (such as Hecke pairs) under these equivalences. In this respect, so far, the following theorem was proved in \cite{j2}.

\begin{theorem}
\label{thm:comgroups}
(Paul Jolissaint) Property (RD) is invariant under weak commensurability of groups.
\end{theorem}

\begin{proof}
Let $H$ be a subgroup of a group $G$ of finite index. It is enough to show that $G$ has (RD) if and only if $H$ has (RD). These implications were proved in Proposition 2.1.1 and Proposition 2.1.5 of \cite{j2}.
\end{proof}

In this paper we study the invariance of property (RD) under different commensurability relations between subgroups in Hecke pairs. Our main result is Theorem \ref{thm:scom} which asserts that if $\gh$ is a Hecke pair with (RD) \wrt a length function $L$ and if $K$ is a subgroup of $G$ strongly commensurable with $H$ such that
\begin{equation}
\label{lengthcondition}
K\subseteq N_L:=\{ g\in G; L(g)=0\},
\end{equation}
then the Hecke pair $\gk$ has (RD) \wrt $L$. The above theorem is extended to the case that $H$ and $K$ are commensurable subgroups of $G$ in Corollary \ref{cor:comconj}. By assuming that one of the subgroups $H$ or $K$ is normal in $G$, one can drop the above condition, (\ref{lengthcondition}), about the length function. Therefore, in Corollary \ref{cor:normal}, it is shown that if $H$ is normal and $K$ is commensurable to $H$, then the quotient group $G/H$ has (RD) if and only if the Hecke pair $(G,K)$ has (RD). In Proposition \ref{prop:finker}, we show how group homomorphisms with finite kernels can be used to pull back property (RD). A similar result about pushing forward property (RD) by surjective homomorphisms is proved in Proposition \ref{prop:surj}. The proof of Theorem \ref{thm:comgroups} relies on the behavior of property (RD) in group extensions. Therefore, to extend this result to the framework of Hecke pairs, we study certain Hecke pairs coming from group extensions as well as the conditions which imply property (RD) for this Hecke pairs in Section \ref{sec:RDext}. So far all examples of Hecke pairs with (RD) discussed in \cite{s1} were restricted to those Hecke pairs $\gh$ that $G$ has (RD) and $H$ is a finite subgroup of $G$. The above results can be applied to prove property (RD) for Hecke pairs $\gh$ which $H$ is an infinite almost normal subgroup of $G$. For example, we find a class of groups all whose subgroups are almost normal and all such Hecke pairs have (RD). Another class of examples of Hecke pairs is also given in Example \ref{ex:examples}.

We conclude this section with recalling some notations. Let $H$ be a subgroup of a group $G$, the index of $H$ in $G$ is denoted by $|G:H|$. For $g\in G$, the number of left (resp. right) cosets of $H$ in the double coset $HgH$ is denoted by $L(g)$ (resp. $R(g)$). Therefore $H$ is almost normal in $G$ if $L(g)<\infty$ for all $g\in G$ (or equivalently, $R(g)<\infty$ for all $g\in G$). We have $L(g)=|H: H\cap gHg\inv|$ and $R(g)=|H: H\cap g\inv Hg|$, and so $L(g)=R(g\inv)$, see Page 170 of \cite{tzanev}. The reader should be careful as we use the symbol $L$ for length functions too. The set of all double cosets of a Hecke pair $\gh$ is denoted by $G//H$. An arbitrary set of representatives of right cosets (resp. double cosets and left cosets) of $H$ in $G$ is denoted by $<H\ba G>$ (resp. $<G//H>$ and $<G/H>$). The vector space of all complex functions on the set $H\ba G$ of right cosets with finite support is denoted by $\c(H\ba G)$. The subsets of non-negative real functions in $\H\gh$, $\c(H\ba G)$, and $\c G$ are denoted by $\H_+\gh$, $\r_+(H\ba G)$, and $\r_+ (G)$, respectively. A length function on a Hecke pair $\gh$ is a length function $L$ on $G$ such that $H\subseteq N_L:=\{ g\in G; L(g)=0\}$. It follows from this latter condition that $L$ is constant on each double coset. Thus, for every non-negative real number $r$, we are allowed to define
\[
B_{r,L}\gh:= \{ HgH\in <G//H>; L(g)\leq r\}.
\]
We also denote similar sets in $<H\ba G>$ and $G$ by $B_{r,L}(H\ba G)$ and $B_{r,L}(G)$, respectively. For $f\in \H\gh$, the operator norm of $\lambda(f)$ is denoted  by $\|\lambda(f)\|$ and, for $s\geq 0$, the weighted $\ell^2$-norm of $f$ with respect to $L$ is defined by
\[
\|f\|_{s,L}:=\left( \sum_{g\in <H\ba G>} \mid f(g)\mid^2
(1+L(g))^{2s}\right)^{\frac{1}{2}}.
\]
For $f\in \c(H\ba G)$ or $f\in \H\gh$, the norm of $f$ in $\ell^2(H\ba G)$  is denoted by $\|f\|_2$.
\begin{definition}
\label{def:rd} Let $L$ be a length function on a Hecke pair $\gh$. We say $\gh$ has  property (RD) \wrt $L$ if following equivalent conditions hold:
\begin{itemize}
\item[(i)] There are positive real numbers $C$ and $s$ such that the Haagerup inequality;
\begin{equation}
\label{def:rd1}
\parallel \lambda(f) \parallel\leq C\| f \|_{s,L}
\end{equation}
holds for all $f\in \H(G,H)$.
\item[(ii)] There exists a polynomial $P$ such that for any $r>0$, $f\in \H_+\gh$ so that $\text{supp}f\subseteq B_{r,L}\gh$, and $k\in \r_+(H\ba G)$, we have
\begin{equation}
\label{def:rd2}
\|f\ast k\|_2\leq P(r)\|f\|_{2}\|k\|_{2}.
\end{equation}
\end{itemize}
\end{definition}

Item (i) in the above is the original definition of property (RD). The equivalence of these conditions was shown in Proposition 2.10 of \cite{s1}. Let $L_1$ and $L_2$ be two length functions on some Hecke pair. We say $L_1$ dominates $L_2$ if there exist positive real numbers $a,b$ such that $L_2\leq a L_1+b$. The length functions $L_1$ and $L_2$ are called equivalent if they dominate each other.
\begin{remark}
\label{r:eqlen1}
If a Hecke pair has (RD) \wrt a length function $L$ it has (RD) \wrt all length functions dominating $L$, in particular those which are equivalent to $L$.
\end{remark}

%%%%%%%%%%%%%%%%%%%%%%%%%%%%%%%%%%%%%%%%%%%%%%%%%%%%%%%%%%%%%%%%%%%%%%%%%%%%%%%%%%%%%%%%%%%%%%%%%%%%%%%%%%%
%%%%%%%%%%%%%%%%%%%%%%%%%%%%%%%%%%%%        SECTION      property (RD)        %%%%%%%%%%%%%%%%%%%%%%%%%%%%%
%%%%%%%%%%%%%%%%%%%%%%%%%%%%%%%%%%%%%%%%%%%%%%%%%%%%%%%%%%%%%%%%%%%%%%%%%%%%%%%%%%%%%%%%%%%%%%%%%%%%%%%%%%%

\section{Property (RD) and commensurable subgroups of Hecke pairs}
\label{sec:commsub}

\begin{remark}
\label{r:comm1} If $H$ and $K$ are two commensurable subgroup of $G$ and $H$ is almost normal, then $gKg^{-1}$ is almost normal in $G$ for some $g\in G$, see Page 170 of \cite{tzanev}, and this implies that $K$ is almost normal in $G$.
\end{remark}
\begin{theorem}
\label{thm:scom}
Let $H$ and $K$ be two strongly commensurable and almost normal subgroups of a group $G$. If there exists a length function $L$ on $G$ such that $H, K\subseteq N_L$, then the Hecke pair  $\gh$ has (RD) \wrt $L$ if and only if the Hecke pair $\gk$ has (RD) \wrt $L$.
\end{theorem}

\begin{proof}
By replacing $K$ with $K\cap H$, without loss of generality, we can assume $K$ is a subgroup of $H$ of finite index $n$.

Suppose $\gk$ has (RD) \wrt $L$ and $P$ is the polynomial appearing in Definition \ref{def:rd}. Let $f\in \H_+\gh$ with $\text{supp} f\subseteq B_{r,L}\gh$ and let $k\in\r_+(H\ba G)$. We define $\tilde{k}\in \r_+(K\ba G)$ (resp. $\tilde{f} \in \H_+\gk$) by $\tilde{k}(x)=\tilde{k}(Kx):=k(Hx)=k(x)$  (resp. $\tilde{f}(x)=\tilde{f}(KxK):=f(HxH)=f(x)$) for all $x\in G$. One notes that $\tilde{f}\in B_{r,L}\gk$. On the other hand, since every right coset of $H$ is  the disjoint union of exactly $n$ right cosets of $K$, we have
\[
\|\tilde{k}\|_2^2=n \|k\|_2^2,\quad \text{and}\quad \|\tilde{f}\|_2^2=n \|f\|_2^2,
\]
where the norms are taken in $\ell^2(K\ba G)$ and $\ell^2(H\ba G)$, accordingly. If $Hx=\bigcup_{i=1}^n Kx_i$ (which implies that $k(x)=\tilde{k} (x_i)$ for all $i=1,\cdots, n$), then $yx\inv\in \bigcup_{i=1}^n yx_i\inv KH=\bigcup_{i=1}^n yx_i\inv H$, and so $Hyx\inv\subseteq \bigcup_{i=1}^n H yx_i\inv H$. This means $f(yx\inv) =\tilde{f}(yx_i\inv)$ for all $i=1,\cdots, n$. Hence, we have

\begin{eqnarray*}
\|\tilde{f}\ast\tilde{k}\|_2^2 &=& \sum_{y\in <K\ba G>}\left(\sum_{x\in <K\ba G>} \tilde{f}(y x^{-1})\tilde{k}(x)\right)^2\\
&=& \sum_{y\in <K\ba G>}\left(\sum_{x\in <H\ba G>} \left[\sum_{i=1,\, Hx=\bigcup_{i=1}^n Kx_i}^n \tilde{f}(y x_i^{-1})\tilde{k}(x_i)\right] \right)^2 \\
&=&\sum_{y\in <K\ba G>}\left(\sum_{x\in <H\ba G>} n f(yx\inv ) k (x) \right)^2 \\
&=&\sum_{y\in <K\ba G>} n^2 (f\ast k(y))^2\\
&=& n^3 \|f\ast k\|^2_2.
\end{eqnarray*}
Now, we compute $\|f\ast k\|_2^2= \frac{1}{n^3}\|\tilde{f}\ast\tilde{k}\|_2^2\leq \frac{1}{n^3}P(r)^2 \|\tilde{f}\|_2^2\|\tilde{k}\|_2^2 \leq \frac{1}{n} P(r)^2 \|f\|_2^2 \|k\|_2^2$. Thus $\gh$ has (RD) \wrt $L$.

Conversely, assume $\gh$ has (RD) \wrt $L$ and let $P$ be the polynomial in Definition \ref{def:rd}. Let $\{ h_1, \cdots, h_n\}$ be a complete set of representatives of right cosets of $K$ in $H$. For $f\in \H_+\gk$ with $\text{supp}f \subseteq B_{r,L}\gk$, define $\bar{f}\in \H_+\gh$ by $\bar{f}(x)=\bar{f}(HxH):=\sum_{i,j=1}^n f(Kh_i x h_j K)=\sum_{i,j=1}^n f(h_i x h_j)$ for all $x\in G//H$. For $m\in\n$, let $c(m)$ be the least constant for which $\left(\sum_{i=1}^m a_i\right)^2\leq c(m)
\sum_{i=1}^m a_i^2$ for all $a_i\geq 0$. One computes
\begin{eqnarray*}
\|\bar{f}\|_2^2 &=&\sum_{x\in <H\ba G>} \left( \sum_{i,j=1}^n f(h_i x h_j)  \right)^2\\
&\leq& c(n^2) \sum_{x\in <H\ba G>} \sum_{i,j=1}^n \left( f(h_i x h_j)  \right)^2\\
&\leq& n^2c(n^2) \| f\|_2^2,
\end{eqnarray*}
where the last inequality follows from the fact that every right coset of $K$ appears at least once and at most $n^2$ times in the last summation. For $k\in \r_+ (K\ba G)$, define $\bar{k}\in\r_+(H\ba G)$ by $\bar{k} (g)=\bar{k} (gH):= \sum_{i=1}^n k(Kh_ig)=\sum_{i=1}^n k(h_ig)$. A similar computation as above shows that $\|\bar{k}\|_2^2 \leq n c(n)\| k\|_2^2$. We also note that $f\leq \tilde{\bar{f}}$ and $k\leq \tilde{\bar{k}}$.
Therefore, for every $g\in <K\ba G>$, we have $f\ast k (g)=\sum_{x\in <K\ba G>}f(gx\inv)k(x) \leq \sum_{x\in <K\ba G>}\tilde{\bar{f}}(gx\inv)\tilde{\bar{k}}(x)=\tilde{\bar{f}} \ast \tilde{\bar{k}}(g)$. Hence
\begin{eqnarray*}
\|f\ast k\|_2^2&\leq & \| \tilde{\bar{f}}\ast \tilde{\bar{k}}\|_2^2\\
&= & n^{3} \|\bar{f}\ast \bar{k}\|_2^2\\
&\leq & n^{3} P(r)^2 \|\bar{f}\|_2^2\| \bar{k}\|_2^2\\
&\leq & n^6 c(n)c(n^2) P(r)^2 \|f\|_2^2\| k\|_2^2.
\end{eqnarray*}
This shows that $\gk$ has (RD) \wrt $L$ and completes the proof.
\end{proof}

The proof of this theorem is a modification of the proof of Theorem 2.11 of \cite{s1}. The only weakness of the above theorem is that it assumes the existence of the length function $L$ such that $H, K \subseteq N_L$. This difficulty can be resolved when one of the subgroups $H$ or $K$ is normal in $G$. First we need some basics about extensions of groups.

\begin{remark}
\label{r:ext}
Let $H$ be a normal subgroup of $G$ and set $Q:=G/H$. Then we can consider $G$ as the extension of $Q$ by $H$;
\begin{equation}
\label{diag:gpext}
\xymatrix{
1\ar[r]&H\ar[r]&  G\ar[r]^\pi & Q\ar[r] & 1.
}
\end{equation}
As it was explained at the beginning of Section 2 of \cite{j2}, one equips  $H\times Q$ with the multiplication defined by the formula
\begin{equation}
\label{f:extprod}
(h_1,x_1)(h_2, x_2) := (h_1\rho(x_1)(h_2)f(x_1,x_2), x_1x_2),\, \forall (h_1,x_1), (h_2, x_2)\in H\times Q,
\end{equation}
where $f$ and $\rho $ are defined as follows: First, we fix a set-theoretic cross-section $\sigma:Q\ra G$ of $\pi$ such that $\sigma(1_Q)=1_G$. Then we define $f:Q\times Q\ra H$ by $f(x_1,x_2):= \si(x_1)\si(x_2)\si(x_1x_2)\inv $. For every $x\in Q$, $\rho(x)$ is the automorphism of $H$ defined by conjugation of $\si(x)$, that is $\rho(x)(h)=\si(x) h \si(x)\inv$. We also note that $f$ and $\rho$ are related as follows:
\begin{equation}
\label{f:extrel1}
f(x_1,x_2)f(x_1x_2, x_3)=\rho(x_1)(f(x_2, x_3))f(x_1, x_2 x_3)
\end{equation}
and
\begin{equation}
\label{f:extrel2}
\rho(x_1)\rho(x_2)=Ad(f(x_1, x_2))\rho(x_1 x_2),
\end{equation}
for all $x_1, x_2, x_3 \in Q$, where $Ad(h)$ is the inner automorphism of $H$ defined by $h$ for all $h\in H$, i.e. $Ad(h)(a)=hah\inv$ for all $a\in H$, see Page 104 of \cite{brown}. The multiplication defined by \ref{f:extprod} is associative and one easily checks that the inverse of an element $(h,x)\in H\times Q $ \wrt this multiplication is given by the following formula:
\begin{equation}
\label{f:extinv}
(h,x)\inv = (\si(x)\inv h\inv \si(x\inv )\inv , x\inv).
\end{equation}
Then the map $H\times Q\ra G$ defined by $(h,x)\mapsto h\si(x)$ is a group isomorphism. Since our decomposition of $G$ depends on $\si$, we denote the group $H\times Q$ by $G_\si$.
\end{remark}

\begin{lemma}
\label{lem:length}
Let $H$ be a normal subgroup of a group $G$ and let $K$ be a subgroup of $G$ which is commensurable with $H$. If $L$ is a length function on either the Hecke pair $\gh$ or the Hecke pair $(G,K)$, then there exists a length function $L^\prime$ on $G$ equivalent with $L$ such that $N_{L^\prime}$ contains both $H$ and $K$.
\end{lemma}
\begin{proof} One notes that $K$ and $H$ are strongly commensurable, because $H$ is normal in $G$. As the first case, assume that $K\subseteq N_L$. We consider the group extension (\ref{diag:gpext}) and constructions described in Remark \ref{r:ext}. Using the isomorphism $G_\si=H\times Q\ra G$, we consider $L$ as a length function over $G_\si$. We define $L^\prime:G_\si \ra [0,\infty[$ by $(h,x)\mapsto L(1_H, x)$ for all $(h,x)\in G_\si$. One easily checks that $L^\prime$ is a length function on $G_\si$. One also notes that $m:=|H:H\cap N_L|<\infty$. Let $<(H\cap N_L)\ba H>=\{ (H\cap N_L) h_1, \cdots, (H\cap N_L) h_m \}$ and set $M=\max\{L(h_i, 1_Q); 1\leq i \leq m\}$. Then, for all  $(h,x)\in G_\si$, we have $L(h,x)=L((h,1_Q)(1_H,x))\leq L(h,1_Q)+L(1_H,x)\leq M+L^\prime(h,x)$ and similarly $L^\prime(h,x)-M\leq L(h,x)$. Therefore $L$ and $L^\prime$ are equivalent. It is also clear that $H, K \subseteq N_{L^\prime}$.

As the second case, assume that $H\subseteq N_L$. By Remark 2.3 of \cite{s1}, $L$ is constant on each left coset of $H$, so it can be considered as a length function on $G/H$. Since $H$ and $K$ are strongly commensurable the quotient group $\frac{HK}{H}$ is a finite subgroup of $G/H$. It follows from Lemma 2.1.3 of \cite{j2} that there exists a length function $\tilde{L}$ on $G/H$ equivalent to $L$ such that $\frac{HK}{H}\subseteq N_{\tilde{L}}$. Define $L^\prime:G\ra [0,\infty[$ by $L'(g):=\tilde{L}(gH)$ for all $g\in G$. The function $L^\prime$ is a length function on $G$ equivalent to $L$ and $N_{L^\prime}$ contains both $H$ and $K$.
\end{proof}

\begin{corollary}
\label{cor:normal}
Let $H$ and $K$ be two commensurable subgroups of a group $G$ and let $H$ be normal in $G$. Then the quotient group $G/H$ has (RD) if and only if the Hecke pair $\gk$ has (RD).
\end{corollary}
\begin{proof}
Since property (RD) of $G/H$ is the same as property (RD) of the Hecke pair $\gh$, the statement follows from Theorem \ref{thm:scom} and Lemma \ref{lem:length}.
\end{proof}

Now, using the above corollary, we are able to introduce a class of groups all whose subgroups are almost normal and every such Hecke pair has property (RD). First we need some definitions and results from group theory. Our main reference is the paper \cite{neu} by B. H. Neumann.

\begin{definition} Let $H$ be a subgroup of a group $G$. The normal closure of $H$ in $G$ is the smallest normal subgroup of $G$ containing $H$. It is denoted by $H^G$. $H$ is called nearly normal in $G$ if $|H^G:H|<\infty$.
\end{definition}
The following classes of groups are important for our discussion.
\begin{definition} Let $G$ be a group.
\begin{itemize}
\item [(i)] It is called an FD-group or a commutator-finite group if its commutator (derived) subgroup, $G^\prime$, is finite.
\item [(ii)] It is called an FIZ-group if the group of inner automorphisms of $G$, $Inn(G)=\frac{G}{Z(G)}$, is finite, where $Z(G)$ denotes the center of $G$.
\item [(iii)] It is called an FC-group if all conjugacy classes of elements of $G$ are finite or equivalently the centralizer of every element of $G$ is a subgroup of finite index.
\end{itemize}
\end{definition}

In the following theorem we summarize those results of \cite{neu} that we need for our purpose.
\begin{theorem} (B. H. Neumann).
\begin{itemize}
\item [(i)] Every FIZ-group is an FD-group and every FD-group is an FC-group.
\item [(ii)] In the class of finitely generated groups the converses of the above implications are also true. In other words, if $G$ is a finitely generated FC-group, then it is an FIZ-group.
\item [(iii)] Subgroups and quotients of an FD-group, (resp. FIZ-group), (resp. FC-group) are FD-groups, (resp. FIZ-groups), (resp. FC-groups).
\item [(iv)] A group $G$ is an FD-group if and only if every subgroup of $G$ is a nearly normal subgroup of $G$.
\end{itemize}
\end{theorem}

%It is clear from definition and Part (iv) of the above theorem that every subgroup of an FD-group $G$ is almost normal in $G$.
The following corollary relates the above theorem to our discussion:

\begin{corollary} Let $G$ be a finitely generated FD-group (or FIZ-group, or FC-group) and let $H$ be a subgroup of $G$. Then the Hecke pair $\gh$ has (RD).
\end{corollary}
\begin{proof}
First we prove that every finitely generated FD-group $\gaa$ has (RD). By definition the commutator subgroup $\gaa^\prime$ of $\gaa$ is finite. On the other hand, $\gaa/\gaa^\prime$ is a finitely generated abelian group and has (RD), because it is isomorphic to a direct product of finitely many cyclic (finite or infinite) groups which all have (RD), see Example 1.2.3 of \cite{j2} and Example 2.8 of \cite{s1}. Therefore, by Proposition 2.1.4 of \cite{j2} (or by Theorem \ref{thm:scom}), $\gaa$ has (RD).

Let $H$ be a subgroup of $G$ then $G/H^G$ is a finitely generated FD-group and so it has (RD). Now, since $|H^G:H|<\infty$, $H^G$ and $H$ are strongly commensurable. Therefore, by Corollary \ref{cor:normal}, the Hecke pair $\gh$ has (RD).
\end{proof}

One can use Corollary \ref{cor:normal} to construct more multifarious examples of Hecke pairs with property (RD). We use the fact that if $A$ and $B$ are two groups, the kernel of the natural surjective homomorphism $\pi:A\ast B\ra A\times B$ is a free group generated by commutators $[a,b]$ for all $a\in A-\{1_A\}$ and $b\in B-\{1_B\}$, see Proposition 4 in Chapter 1 of \cite{serre}. Denote this kernel by $R_{A,B}$.

\begin{example}
\label{ex:examples}
Pick two groups $A$ and $B$ having property (RD). Then the quotient group $\frac{A\ast B}{R_{A,B}}\simeq A\times B$ has (RD). For any subgroup $H$ of $A\ast B$ commensurable with $R_{A,B}$, the Hecke pair $(A\ast B, H)$ has (RD). There are basically two ways to construct subgroups like $H$ commensurable with $R_{A,B}$. First, one can pick any finite subgroup of $A\times B$, say $K$, and consider its preimage in $A\ast B$. Clearly, $\pi\inv (K)$ contains $R_{A,B}$ as a finite index subgroup. Secondly, since $R_{A,B}$ is a free group, for every integer $n$, one can define a surjective homomorphism $R_{A,B}\ra \z/n\z$ by mapping one of the generators to $1\in \z/n\z$ and mapping other generators to zero. The kernel of this homomorphism is a finite index subgroup of $R_{A,B}$. One should note that the subgroups of $A\ast B$ commensurable with $R_{A,B}$, constructed in the above, are not necessary normal in $A\ast B$. Our discussion gives us a huge freedom to choose $A$ and $B$ and construct Hecke pairs with property (RD). Obviously, one can also apply our discussion to construct non-elementary Hecke pairs which do not have property (RD).
\end{example}

The rest of this section is devoted to show how one can pull back or push forward property (RD) using specific types of group homomorphisms.

\begin{proposition}
\label{prop:finker}
Let $\f:G_1\ra G_2$ be a group homomorphism whose kernel is finite. If $H_2\subseteq \f(G_1)$ is an almost normal subgroup of $G_2$ and the Hecke pair $(G_2,H_2)$ has (RD), then the Hecke pair $(G_1,\f^{-1}(H_2))$ has (RD).
\end{proposition}
\begin{proof} It was shown in Page 170 of \cite{tzanev} that $H_1:=\f^{-1}(H_2)$ is almost normal in $G_1$. Assume $(G_2,H_2)$ has (RD) \wrt a length function $L_2$ and $P$ is the polynomial in Definition \ref{def:rd}. Define $L_1:=L_2 \f$. Then $L_1$ is a length function on the Hecke pair $(G_1,H_1)$. For $f_1\in \H_+(G_1,H_1)$ with $\text{supp}f_1\subseteq B_{r,L_1}(G_1,H_1)$ and $k_1\in \r_+(H_1\ba G_1)$, define $f_2\in \H_+(G_2,H_2)$ by $f_2(x):=f_1(y)$ if there exists some $y\in G_1$ such that $x=\f(y)$ and otherwise define $f_2(x):=0$. Similarly define $k_2\in \r_+(H_2\ba G_2)$. Since $Ker(\f)\subseteq H_1$, the definition of $f_2(x)$ and $k_2(x)$ does not depend on the pre-image of $x$ and they are well-defined. For instance,  we check the invariance of $f_2$ under multiplication of elements of $H_2$ from left. For $h_2\in H_2$, choose $h_1\in H_1$ such that $\f(h_1)=h_2$. Then, for all $x\in \f(G_1)$, if $h_2x =\f(y)$, then $x=\f(h_1^{-1}y)$ and so $f_2(h_2x)=f_1(y)=f_1(h_1^{-1}y)=f_2(x)$ as claimed. It is also easy to see that $\text{supp}f_2\subseteq B_{r,L_2}(G_2,H_2)$. We compute $\|f_2\|_2^2=\sum_{x\in <H_2\ba G_2>} (f_2(x))^2=\sum_{y\in <H_1\ba G_1>} (f_1(y))^2=\|f_1\|_2^2$ and similarly $\|k_2\|_2^2=\|k_1\|_2^2$. If $x=\f(y)$ and $s=\f(t)$, then $sx^{-1}=\f(ty^{-1})$. Thus we have
\begin{eqnarray*}
\|f_2\ast k_2\|_2^2&=&\sum_{s\in <H_2\ba G_2>} \left( \sum _{x\in <H_2\ba G_2>} f_2(s x\inv)k_2(x) \right)^2\\
&=&\sum_{t\in <H_1\ba G_1>} \left( \sum_{y\in <H_1\ba G_1>} f_1(t y\inv)k_1(y) \right)^2\\
&=&\|f_1\ast k_1\|_2^2.
\end{eqnarray*}
The above computations are based on the fact that a right coset of $H_2$ either has a pre-image which has to be unique or it has no pre-image and in the latter case, $f_2$ and $k_2$ at this right coset have to be zero. From these equalities we conclude that $\|f_1\ast k_1\|_2^2\leq P(r) \|f_1\|_2^2\|k_1\|_2^2$, and so $(G_1,H_1)$ has (RD) \wrt $L_1$.
\end{proof}

Two immediate corollaries of the above proposition are as follows:

\begin{corollary}
\label{cor:subHecke}
Let $H$ be an almost normal subgroup of a group $G$ and let $\Gamma$ be a subgroup of $G$ containing $H$. If the Hecke pair $\gh$ has (RD), then the Hecke pair $(\Gamma, H)$ has (RD).
\end{corollary}
One notes that the above corollary is a generalization of Proposition 2.1.1 of \cite{j2}.
\begin{corollary}
\label{cor:conjugate}
Let $H$ and $K$ be two almost normal subgroups of a group $G$ which are conjugate, namely there exists $g\in G$ such that $K=gHg^{-1}$. Then the Hecke pair $\gh$ has (RD) if and only if the Hecke pair $\gk$ has (RD).
\end{corollary}

The invariance of property (RD) of Hecke pairs under commensurability of subgroups follows immediately from the above corollary and Theorem \ref{thm:scom}.

\begin{corollary}
\label{cor:comconj}
Let $H$ and $K$ be two almost normal subgroups of a group $G$. Assume there exists $g\in G$ such that $gKg^{-1}$ is strongly commensurable with $H$. If there is a length function $L$ such that $ gKg^{-1}\cup H\subseteq N_L$, then the Hecke pair $\gh$ has (RD) \wrt $L$ if and only if $\gk$ has (RD) \wrt $L$.
\end{corollary}

The following proposition is some how the dual of Proposition \ref{prop:finker}.

\begin{proposition}
\label{prop:surj}
Let $\f:G_1\ra G_2$ be a surjective group homomorphism and let $H_1$ be an almost normal subgroup of $G_1$ containing $Ker\f$. If the Hecke pair $(G_1,H_1)$ has (RD), then $(G_2,\f(H_1))$ is a Hecke pair with property (RD).
\end{proposition}
\begin{proof} Set $H_2:=\f(H_1)$. It is shown in page 170 of \cite{tzanev} that $(G_2,H_2)$ is a Hecke pair. Assume that $(G_1,H_1)$ has (RD) \wrt $L_1$ and $P$ is the polynomial in Definition \ref{def:rd}. For given $ y\in G_2$, let $x$ be an element of $G_1$ such that $y=\f(x)$ and define $L_2(y):=L_1(x)$. Since $Ker\f\subseteq H_1\subseteq N_{L_1}$, the function $L_2:G_2\ra [0,\infty[$ is well-defined. Indeed, it is a length function on the Hecke pair $(G_2,H_2)$. For $f_2\in \H_+(G_2,H_2)$ with $\text{supp}f_2\subseteq B_{r,L_2}(G_2,H_2)$ and $k_2\in \r_+(H_2\ba G_2)$, we define $f_1:=f_2 \f$ and  $k_1:=k_2 \f$. It follows from the definition that $f_1\in \H_+(G_1,H_1)$ with $\text{supp}f_1\subseteq B_{r,L_1}(G_1,H_1)$ and $k_1\in \r_+(H_1\ba G_1)$. In fact, $\f$ induces a bijection between the set of right cosets and a bijection between the set of double cosets of the Hecke pairs $(G_1,H_1)$ and $(G_2,H_2)$. It follows from these bijections that $\|f_2\|_2^2=\|f_1\|_2^2$, $\|k_2\|_2^2=\|k_1\|_2^2$ and $\|f_1\ast k_2\|_2^2=\|f_1\ast k_1\|_2^2$. The rest of the proof is straightforward.
\end{proof}

We note that Propositions \ref{prop:finker} and \ref{prop:surj} together are a generalization of Proposition 2.1.4 of \cite{j2}. Note that the condition $Ker\f\subseteq H$ is necessary in the above proposition, as it is apparent from the proof. Otherwise, since every finitely generated group is the quotient of a free group of finite rank, it would have (RD). However the condition $\f$ being surjective can be relaxed in some cases that will be discussed in the following section.

%%%%%%%%%%%%%%%%%%%%%%%%%%%%%%%%%%%%%%%%%%%%%%%%%%     SECTION        %%%%%%%%%%%%%%%%%%%%%%%%%%%%%%%%%%%%%%%%%%%%%%%%
\section{Property (RD) and extensions of groups by Hecke pairs}
\label{sec:RDext}

In this section we consider a group extension
\begin{equation}
\label{diag:gpext2}
\xymatrix{
1\ar[r]&G \ar[r]&  E \ar[r]^\pi & \gaa \ar[r] & 1
}
\end{equation}
and an almost normal subgroup $H$ of $G$ and study when $H$ is an almost normal subgroup of $E$ and when property (RD) of the Hecke pair $(E,H)$ follows from property (RD) of the Hecke pair $\gh$ and the group $\gaa$.

\begin{definition}
\label{def:conext} In the above situation, we call a set-theoretic cross-section $\si: \gaa\ra E$ consistent with the Hecke pair $\gh$, if $\rho(\ga)(H)=\si(\ga) H\si(\ga)\inv=H$ for all $\ga\in \gaa$. We also call the group extension \ref{diag:gpext2} consistent with the Hecke pair $\gh$ if there is a set-theoretic cross-section $\si:\gaa\ra E$ consistent with $\gh$.
\end{definition}

\begin{remark}
\begin{itemize}
\item [(i)] The group extension (\ref{diag:gpext2}) is split, namely there exist a homomorphism $s:\gaa\ra E$ such that $s$ is a cross-section if and only if $E$ is semidirect product of $G$ and $\gaa$. In this case, the condition $s(\ga)Hs(\ga)\inv\subseteq H$ for all $\ga\in \gaa$ implies that $s$ is a cross-section consistent with the Hecke pair $\gh$.
\item [(ii)] If a cross-section $\si$ is consistent with the Hecke pair $\gh$, then the number, $L(\si(\ga))$, of distinct left cosets of $H$ in $H\si(\ga)H$ equals 1 for all $\ga\in \gaa$.
\end{itemize}
\end{remark}

\begin{lemma}
\label{lem:conext}
Assume (\ref{diag:gpext2}) is a group extension consistent with the Hecke pair $\gh$ and $\si$ is the consistent cross-section described in the above. Let $f$ and $\rho$ be as defined in Remark \ref{r:ext}.
\begin{itemize}
\item [(i)] The pair $(E,H)$ is a Hecke pair.
\item [(ii)] If $H$ is normal in $G$, then it is normal in $E$ too.
\item [(iii)] For every $\beta , \ga\in \gaa$, the map $\theta_{\ga,\beta}:<G/H>\ra <G/H>$ defined by $g\mapsto \si(\ga)\si(\beta)\inv g \si(\ga \beta\inv)\inv$ for all $g\in G$ is a well-defined bijective map.
\item [(iv)] The map $(Hg, \ga)\mapsto H(g,\ga)$ is a bijection between the sets $<H\ba G > \times \gaa$ and $<H\ba (G\times\gaa)>=<H\ba E_\si>$, where the multiplication between elements of $G$ and elements of $\gaa$ in both sides are as in $E_\si$.
\end{itemize}
\end{lemma}
\begin{proof}
\begin{itemize}
\item [(i)] Let $g\in G$ and $h\in H\cap gHg\inv$. Then there exists $h_1\in H$ such that $hg=gh_1$. For given $\ga\in \gaa$, set $h_2:= \si(\ga)\inv h_1\si(\ga)$. Then $h_2\in H$ and we have $(g,\ga)(h_2,1_\gaa)=(g\si(\ga)h_2 \si(\ga)\inv ,\ga)=(gh_1, \ga)=(hg,\ga)= (h,1_\gaa) (g,\ga)$. This shows that $h\in (g,\ga)H(g,\ga)\inv$, so $H\cap gHg\inv\subseteq H\cap (g,\ga)H(g,\ga)\inv$. Hence
    \begin{equation}
    \label{f:inequalities}
    |H: H\cap (g,\ga)H(g,\ga)\inv|\leq |H: H\cap gHg\inv|<\infty.
    \end{equation}
    Since this is true for every $g\in G$ and $\ga\in \gaa$, $H$ is almost normal in $E_\si$ and so is in $E$.
\item [(ii)] It is clear from (\ref{f:inequalities}).
\item [(iii)] One easily computes
\[
\theta_{\ga,\beta} (g)= \si(\ga)\si(\beta)\inv\si(\ga \beta\inv)\inv \rho(\ga\beta\inv)(g).
\]
Since $\rho(x)$ is an automorphism on $G$ for all $x\in \gaa$ and the group extension is consistent with the Hecke pair, the map defined by $g\mapsto\rho(\ga\beta\inv)(g)$ is a bijective map from $<G/H>$ onto $<G/H>$. On the other hand, the left multiplication of an element (here $\si(\ga)\si(\beta)\inv\si(\ga \beta\inv)\inv $\,) of $G$ in left cosets is a permutation of the set of left cosets. Therefore $\theta_{\ga, \beta}$ is well defined and bijective for every $\ga,\beta\in \gaa$.
\item [(iv)] It is straightforward.
\end{itemize}
\end{proof}

The existence of  a consistent cross-section in a group extension like (\ref{diag:gpext2}) is not the only case that leads to an extension of a group $\gaa$ by a Hecke pair $\gh$. For example, the famous Hecke pair used by Jean-Beno\^{\i}t Bost and Alain Connes in \cite{bc} comes from the following group extension $0\ra \q\ra \q\rtimes \q_+^\times \ra  \q_+^\times \ra 1$. Here $H=\z$ and is a normal subgroup of $\q$ and so the Bost-Connes Hecke pair is $(\q\rtimes \q_+^\times, \z\rtimes 1)$. To see the action of $\q_+^\times$ on $\q$, one needs to consider them as matrix groups; $\q=\left\{ \left( \begin{array} {rr}1&b\\0&1 \end{array}\right); b\in \q \right\}$ and $\q_+^\times=\left\{ \left( \begin{array} {rr}1&0\\0&a \end{array}\right); a\in \q_+^\times \right\}$. In \cite{s2}, we use this realization of the Bost-Connes Hecke pair to show that this Hecke pair does not have (RD). One notes that the above group extension is not consistent with the Hecke pair $(\q,\z)$, because otherwise $\z\rtimes 1$ would be normal in $\q\rtimes \q_+^\times$ by Lemma \ref{lem:conext}(ii). Brenken generalized this example and proved that, regarding the group extension (\ref{diag:gpext2}), if $H$ is a normal subgroup of $G$ and for every $\ga\in \gaa$ the subgroup $\frac{\rho(\ga)(H)H}{H}$ of $G/H$ is finite, then $H$ is an almost normal subgroup of $E$, see Lemma 1.9 in \cite{brenken}. Similar conditions were given in Proposition 1.7 of \cite{ll}, see also \cite{baumg}. Although in the above papers there are weaker conditions which imply $H$ is almost normal in $E$, generally these weaker conditions do not imply Lemma \ref{lem:conext}(iii) which is an important ingredient in the proof of the following proposition.

\begin{proposition}
\label{prop:extrd}
Let $H$ be an almost normal subgroup of a group $G$ and let $ 1\ra G\ra E\ra \gaa \ra 1$ be a group extension consistent with the Hecke pair $\gh$. Assume $L_0$, $L$ and $L_1$ are length functions on the Hecke pairs $\gh$, $(E ,H)$ and the group $\gaa$, respectively, such that:
\begin{itemize}
\item [(i)] the Hecke pair $\gh$ and the group $\gaa$ have (RD) \wrt $L_0$ and $L_1$, respectively, and
\item [(ii)] there are positive constants $c$ and $e$ such that
\begin{equation}
\label{f:lengthcom}
L_0(s) +L_1(\ga) \leq cL(s,\ga)^e,\quad  \forall (s,\ga)\in E_\si,
\end{equation}
where $E_\si$ denotes the decomposition of $E$ as constructed in Remark \ref{r:ext}.
\end{itemize}
Then the Hecke pair $(E,H)$ has (RD) \wrt $L$.
\end{proposition}
\begin{proof} Let $P_0$ and $P_1$ be the polynomials appearing in the definition of property (RD) for the Hecke pair $\gh$ and the group $\gaa$, respectively. Let $\si:\gaa \ra E$ be the cross-section consistent with the Hecke pair $\gh$ and let $\rho:\gaa \ra Aut(G)$ be as defined in Remark \ref{r:ext}. In the following the subgroup $H\times \{1_\gaa\}$ of $E_\si$ is denoted by the same notation as $H$ considered as the subgroup in $G$ or $E$.
Let $\phi\in \H_+(E_\si, H)$ such that  $supp(\phi)\subseteq B_{r,L}(E_\si, H)$ and let $\psi\in \r_+(H\ba E_\si)$.  For all $g\in H\ba G$ and $\beta, \ga\in \gaa$, we define
\[
\phi_{\ga,\beta}(g):=\phi(g,\ga\beta\inv), \qquad \qquad \psi_{\beta}(g):=\psi(g, \beta).
\]
Then $\phi_{\ga, \beta}\in \H_+\gh$ and $\psi_{\beta}\in \r_+(H\ba G)$ and by applying Inequality (\ref{f:lengthcom}), we have $supp(\phi_{\ga,\beta})\subseteq B_{cr^e,L_0}\gh$. Using (\ref{f:extprod}), (\ref{f:extinv}), Lemma \ref{lem:conext}(iii),(iv) and the fact that $t\in <H\ba G>$ if and only if $t\inv \in <G/H>$, we compute
\begin{eqnarray*}
\|\phi\ast \psi\|_2^2 &=& \sum_{(s,\ga)\in <H\ba E_\si>}\left( \sum_{(t,\beta) \in <H\ba E_\si>} \phi(s\theta_{\ga,\beta}(t\inv),\ga\beta\inv) \psi(t, \beta) \right)^2\\
&=& \sum_{(s,\ga)\in <H\ba G > \times \gaa}\left( \sum_{(t,\beta) \in <H\ba G > \times \gaa} \phi(st\inv,\ga\beta\inv) \psi(t, \beta) \right)^2\\
&=& \sum_{\ga\in \gaa} \sum_{s\in <H\ba G>}\left( \sum_{\beta\in \gaa}\sum_{t \in <H\ba G>} \phi_{\ga,\beta}(st\inv) \psi_\beta (t)\right)^2\\
&=& \sum_{\ga\in \gaa} \left\| \sum_{\beta\in \gaa}\phi_{\ga,\beta}\ast \psi_\beta (s)\right\|_2^2\\
&\leq& \sum_{\ga\in\gaa}\left( \sum_{\beta\in\gaa } \left\| \phi_{\ga,\beta}\ast \psi_\beta\right\|_2  \right)^2\\
&\leq& \sum_{\ga\in\gaa} \left(\sum_{\beta\in\gaa } P_0(cr^e) \| \phi_{\ga,\beta}\|_2\|\psi_\beta\|_2   \right)^2.\\
\end{eqnarray*}
Now, we set $\phi^\prime (\beta):= \left(\sum_{g\in <H\ba G>} \phi(g,\beta)^2\right)^{1/2}$ and $\psi^\prime(\beta):=\|\psi_\beta\|_2$. One notes that $\phi^\prime, \psi^\prime \in \r_+(\gaa)$ and $supp(\phi^\prime)\subseteq B_{cr^e, L_1}(\gaa)$ again by Inequality (\ref{f:lengthcom}). One also easily computes $\|\phi^\prime\|_2=\|\phi\|_2$ and $\|\psi^\prime\|_2=\|\psi\|_2$. Then we have
\begin{eqnarray*}
\phi^\prime(\ga\beta\inv )&=&\left( \sum_{g\in <H\ba G>} \phi(g,\ga\beta\inv )^2 \right)^{1/2}\\
&=& \left( \sum_{g\in <H\ba G>} \phi_{\ga,\beta}(g)^2 \right)^{1/2}\\
&=& \|\phi_{\ga,\beta} \|_2.
\end{eqnarray*}
Hence  we have
\begin{eqnarray*}
\|\phi\ast \psi\|_2^2 &\leq& P_0(cr^e)^2 \sum_{\ga\in\gaa} \left(\sum_{\beta\in\gaa } \phi^\prime (\ga\beta\inv) \psi^\prime(\beta)  \right)^2\\
&=& P_0(cr^e)^2 \| \phi^\prime \ast \psi^\prime \|_2^2\\
&\leq& P_0(cr^e)^2 P_1(cr^e)^2 \| \phi^\prime\|_2^2 \|\psi^\prime \|_2^2\\
&\leq& P(r)^2 \| \phi\|_2^2 \|\psi \|_2^2,
\end{eqnarray*}
where $P$ is a polynomial such that $P_0(cr^e) P_1(cr^e)\leq P(r)$ for all $r>0$.
\end{proof}

The above proposition is a generalization of Lemma 2.1.2 of \cite{j2}. An immediate corollary of the above proposition is that if a Hecke pair $\gh$ and a group $\gaa$ have property (RD), then the Hecke pair $(G\times \gaa, H\times 1_\gaa)$ has (RD). This situation can be generalized further in the form of the following proposition which clearly follows from a similar argument as the proof of the above proposition.

\begin{proposition}
Let $(G_i,H_i)$ be Hecke pairs for $i=1,\cdots, n$. Then the Hecke pair $(\prod_{i=1}^n G_i, \prod_{i=1}^n H_i)$ has (RD) if and only if every Hecke pair $(G_i,H_i)$ has (RD) for $i=1,\cdots, n$.
\end{proposition}

\begin{remark} Let $\gh$ be a Hecke pair. Assume the group extension (\ref{diag:gpext2}) is consistent with the Hecke pair $\gh$. When $\gaa$ is a finite group, it always has property (RD) \wrt the zero length function, which is denoted by $L_1$ here. If the Hecke pair $\gh$ has (RD) \wrt a length function $L_0$, one can use $L_0$ to construct a length function $L$ on $E$ which satisfies Inequality (\ref{f:lengthcom}).

For all $(g,\ga)\in E_\si$, define $p(g,\ga):= g$. One checks that
\[
p((g_1,\ga_1)(g_2,\ga_2))=p(g_1,\ga_1) p((1_G,\ga_1)(g_2,\ga_2)), \quad \forall (g_1,\ga_1), (g_2,\ga_2)\in E_\si.
\]
Now, for all $(g,\ga)\in E_\si$, define $k(g,\ga):=\max_{\beta\in \gaa} L_0(p((1_G,\beta)(g,\ga)))$. Then, for all $(g_1,\ga_1), (g_2,\ga_2)\in E_\si$, we have
\begin{eqnarray*}
k((g_1,\ga_1)(g_2,\ga_2))&=& \max_{\beta\in \gaa} L_0(p((1_G,\beta)[(g_1,\ga_1)(g_2,\ga_2)]))\\
&=& \max_{\beta\in \gaa} L_0(p([(1_G,\beta)(g_1,\ga_1)](g_2,\ga_2)))\\
&=& \max_{\beta\in\gaa} L_0(p((1_G,\beta)(g_1,\ga_1)) p((1_G,\beta \ga_1)(g_2,\ga_2))) \\
&\leq& \max_{\beta\in \gaa} [L_0(p((1_G,\beta)(g_1,\ga_1))) + L_0( p((1_G,\beta \ga_1)(g_2,\ga_2))) ]\\
&\leq& \max_{\beta\in \gaa} L_0(p((1_G,\beta)(g_1,\ga_1))) +  \max_{\delta\in \gaa} L_0( p((1_G,\delta)(g_2,\ga_2))) \\
&=& k(g_1,\ga_1) + k(g_2,\ga_2).
\end{eqnarray*}
Define $L(g,\ga):=k(g,\ga) + k((g,\ga)\inv)$. Clearly $L$ is a length function on $E_\si$ and Inequality (\ref{f:lengthcom}) holds for $c=1$ and $e=1$.

This discussion allows us to drop Condition (ii) about the length functions from Proposition \ref{prop:extrd}.
\end{remark}
The idea applied in the above remark is taken from Proposition 2.1.5 of \cite{j2}. The following corollary is the generalization of Proposition 2.1.5 of \cite{j2} in the setting of Hecke pairs. It is helpful to extend property (RD) to bigger Hecke pairs subject to some conditions. This is done in the following corollary.

\begin{corollary}
Consider a Hecke pair $\gh$ with property (RD) and let $E$ be a group containing $G$ as a subgroup of finite index. Set $G_1:=\bigcap_{x\in E} xGx\inv$ and $H_1:=G_1 \cap H$. The pair $(G_1,H_1)$ is a Hecke pair and has (RD). Furthermore, if the group extension $1\ra G_1\ra E\ra E/G_1\ra 1$ is consistent with the Hecke pair $(G_1,H_1)$, the Hecke pair $(E,H)$ has (RD).
\end{corollary}
\begin{proof}
$G_1$ is a normal subgroup of $E$ of finite index. Thus $H_1$ is a subgroup of $H$ of finite index. This implies that the pair $(G,H_1)$ is a Hecke pair with property (RD). By proposition \ref{prop:finker}, the pair $(G_1,H_1)$ is a Hecke pair with (RD). The rest follows from the above remark and Proposition \ref{prop:extrd}.
\end{proof}

%%%%%%%%%%%%%%%%%%%%%%%%%%%%%%%%%%%%%      Acknowledgements       %%%%%%%%%%%%%%%%%%%%%%%%%%%%%%%%%%%%%%%

\noindent {\bf Acknowledgements.} Example \ref{ex:examples} is the result of a helpful discussion that I have had with Will Sawin about nearly normal subgroups in mathoverflow.net. I would also like to thank the anonymous referee of this paper for suggesting several corrections.

%%%%%%%%%%%%%%%%%%%%%%%%%%%%%%%%%%%%%%%%%%%%%%%%%%%%%%%%%%%%%%%%%%%%%%%%%%%%%%%%%%%%%%%%%%%%%%%%%%%%%%%%%%%%%%%%%%%%%%%%%
%%%%%%%%%%%%%%%%%%%%%%%%%%%%%%%%%%%%%%%%%%%%      bibliography       %%%%%%%%%%%%%%%%%%%%%%%%%%%%%%%%%%%%%%%%%%%%%%%%%%         %%%%%%%%%%%%%%%%%%%%%%%%%%%%%%%%%%%%%%%%%%%%%%%%%%%%%%%%%%%%%%%%%%%%%%%%%%%%%%%%%%%%%%%%%%%%%%%%%%%%%%%%%%%%%%%%%%%%%%%%%
\bibliographystyle {amsalpha}

\begin {thebibliography} {VDN92}

\bibitem{baumg} {\bf U. Baumgartner, J. Foster, J. Hicks, H. Lindsay, B. Maloney, I. Raeburn, S. Richardson,} Hecke algebras and group extensions. Communications in algebra, 33, (2005), 4135--4147.

\bibitem{bc}{\bf J.B. Bost, A. Connes,} Hecke algebras, type III factors and phase transition with spontaneous symmetry breaking in number theory. Selecta Math. (New Series) Vol.{\bf 1}, no.3, 411-456, (1995).

\bibitem{brenken} {\bf B. Brenken,} Hecke algebras and semigroup crossed product \cs-algebras. Pacific Journal of Mathematics, vol. {\bf 187}, no. 2, (1999), 241--262.

\bibitem{brown} {\bf K. S. Brown,} Cohomology of groups. Springer-Verlag, New York, (1982).

%\bibitem{chat} {\bf I. L. Chatterji,} On property (RD) for certain discrete groups. PhD Thesis, Swiss Federal Institute of Technology Zurich, (2001).

%\bibitem{chatru} {\bf I. L. Chatterji, K. Ruane,} Some geometric groups with rapid decay. Geom. funct. anal. vol. {\bf 15}, (2005), 311--339.

\bibitem{harpe} {\bf P. de la Harpe,} Topics in geometric group theory. The University of Chicago Press, (2000).

%\bibitem{j1} {\bf P. Jolissaint,} $K$-theory of reduced \cs-algebras and rapidly decreasing functions on groups. Journal of $K$-Theory {\bf 2:} 723--735, (1989).

\bibitem{j2} {\bf P. Jolissaint,} Rapidly decreasing functions in reduced \cs-algebras of groups. Trans. Amer. Math. Soc. Vol. {\bf 317}, no. 1, 167--196, (1990).

\bibitem{krieg} {\bf A. Krieg,} Hecke algebras. Mem. Amer. Math. Soc. 87 (435), (1990).

\bibitem{ll} {\bf M. Laca, N. S. Larsen,} Hecke algebras of semidirect products. Proc. Amer. Math. Soc. vol. 131, no. 7, (2003), 2189--2199.

\bibitem{luck} {\bf W. L\"{u}ck,} Survey on geometric group theory. M\"{u}nster J. of Math. {\bf 1} (2008), 73--108.

%\bibitem{ls} {\bf R. C. Lyndon, P. E. Shupp,} Combinatorial group theory. Springer-Verlag, (2001).

\bibitem{meier} {\bf J. Meier,} Groups, graphs and trees, an introduction to the geometry of infinite groups. Cambridge University Press, (2008).

\bibitem{neu} {\bf B. H. Neumann,} Groups with finite classes of conjugate subgroups. Math. Zeitschr. Bd. 63, S. (1955), 76--96.

\bibitem{serre} {\bf J-P. Serre,} Trees. Springer-Verlag Berlin Heidelberg (1980).

\bibitem{shimura} {\bf G. Shimura,} Introduction to the Arithmetic Theory of Automorphic Functions. Iwanami Shoten and Princeton University Press, (1971).

\bibitem{s1} {\bf V. Shirbisheh,} Property (RD) for Hecke pairs. Mathematical Physics, Analysis and Geometry, vol. {\bf 15}, no. 2, (2012), 173--192.

\bibitem{s3} {\bf V. Shirbisheh,} An erratum to ``Property (RD) for Hecke pairs''. arXiv:1212.2413, (2012).

\bibitem{s2} {\bf V. Shirbisheh,} Locally compact Hecke pairs, amenability and property (RD). (In preparation).

\bibitem{tzanev} {\bf K. Tzanev,} Hecke \cs-algebras and amenability,  J. Operator Theory, {\bf 50}, (2003), 169--178.

%%%%%%%%%%%%%%%%%%%%%555555555555555555555555555555555555555555555555555555555555555555555555555555555555555555555555555555555555555555555
\end {thebibliography}

\end{document}